\documentclass[12pt,a4paper,draft]{amsart}   
\usepackage[T1]{fontenc}
\usepackage{indentfirst}
\usepackage[arrow,matrix]{xy}
\usepackage{amscd,amssymb,eucal,amsthm}
\usepackage[dvips]{graphicx}

\theoremstyle{plain}
\newtheorem{thm}{Theorem}[section]
\newtheorem{prop}[thm]{Proposition}
\newtheorem{lem}[thm]{Lemma}
\newtheorem{corol}[thm]{Corollary}

\theoremstyle{definition}
\newtheorem{defn}[thm]{Definition}

\newtheorem{examps}[thm]{Examples}

\theoremstyle{remark}
\newtheorem{rem}[thm]{Remark}

\numberwithin{equation}{section} 

\newcommand{\ti}{\text{-}}
\newcommand{\tp}{\otimes}
\newcommand{\F}{\mathbb F}
\newcommand{\Z}{\mathbb Z}
\newcommand{\Rq}{$R$-quadratic\ }
\newcommand{\Rqq}{$R$-quadratic}
\newcommand{\PR}[1]{P^2_R(#1)}
\newcommand{\GA}[1]{\Gamma^2_R(#1)}
\newcommand{\LA}[1]{\Lambda^2_R(#1)}
\newcommand{\ga}[1]{\gamma_2(#1)}
\newcommand{\SY}[1]{{\rm Sym}\/^2_R(#1)}
\newcommand{\mr}[1]{ \stackrel{#1}{\longrightarrow} }
\newcommand{\rond}{{\scriptstyle\circ}}

\newcommand{\noqed}{\renewcommand{\qed}{}}

\DeclareMathOperator{\coker}{Coker}
\DeclareMathOperator{\Ker}{Ker}
\DeclareMathOperator{\IM}{Im}

\newenvironment{MYitemize}
{\begin{itemize}}{\end{itemize}}

\begin{document}

\title{Quadratic maps between modules}
\author[H. GAUDIER]{Henri GAUDIER}
\author[M. HARTL]{Manfred HARTL}
\address{LAMAV and FR CNRS 2956, Universit\'e de Valenciennes, le Mont Houy, \hfil\break \strut\hskip4.3mm F-59313 VALENCIENNES-Cedex~9}
 \email{manfred.hartl@univ-valenciennes.fr}
\email{henri.gaudier@univ-valenciennes.fr}
\date{\today}
\maketitle
\begin{abstract}
 We introduce a notion of $R$-quadratic maps between modules over a commutative ring $R$ which generalizes several classical notions arising in linear algebra and group theory. On a given module $M$ such maps are represented by $R$-linear maps on a certain module $P^2_R(M)$. The structure of this module is described in term of the symmetric tensor square ${\rm Sym}^2_R(M)$, the degree 2 component $\Gamma^2_R(M)$ of the divided power algebra over $M$,  and the ideal $I_2$ of $R$ generated by the elements $r^2-r$, $r\in R$. The latter is shown to represent quadratic derivations on $R$ which arise in the theory of modules over square rings. This allows  to extend the classical notion of nilpotent $R$-group of class 2 with coefficients in a 2-binomial ring $R$ to any ring $R$. We provide a functorial presentation of $I_2$ and several exact sequences embedding the modules $P^2_R(M)$ and $\Gamma^2_R(M)$.
\end{abstract}

In this paper, we introduce and study quadratic maps between modules $M$ et $N$ over a commutative ring $R$ with $1$.
Quadratic forms are the most classical example of such maps; more generally, a notion of homogenous polynomial maps from $M$ to $N$ has been defined in such a way that they are represented by $R$-linear maps from $\Gamma^n_R(M)$ to $N$, where $\Gamma^n_R(M)$ is the homogenous term of degree $n$ of the divided power algebra over $R$ \cite{Rob}. Non-homogenous polynomial maps are then defined to be sums of homogenous ones. This viewpoint is the basis of the recent theory of strict polynomial functors with its numerous spectacular applications, notably allowing to compute the generic cohomology of general linear groups over finite fields \cite{FS}. So this definition of polynomial maps is very satisfactory when $R$ is a field;  
for general rings $R$, however, it is too restrictive: for $R=M=N=\Z$, the 
map assigning $\binom n2$ to $n$ should certainly be considered as being 
quadratic, but does not split as a sum of a linear and a homogenous quadratic map. This example actually comes from group theory where a notion of polynomial maps from groups to   abelian  groups was introduced by Passi \cite{Pa68} in the context of dimension subgroups, but later on turned out to admit many other applications in nilpotent group theory, too \cite{Diss}, \cite{PolProp}, \cite{GoG}.  A more general notion of quadratic maps between arbitrary groups \cite{QmG} arises 
in the new field of ``quadratic algebra" which furnishes an appropriate algebraic framework for dealing with various quadratic phenomena arising in homotopy theory, such as metastable homotopy, secondary homotopy groups and operations, 3-types, quadratic homology etc. This subject is developped in work of Baues, Jibladze, Pirashvili, Muro and the second author, see e.\ g.\ \cite{B}, \cite{BHP}, \cite{BJP1}, \cite{BJP2}, \cite{BM}.

This paper is meant to provide a bridge between the classical realm of quadratic maps and the recent domain of quadratic algebra, with the final aim to apply methods of the latter to problems of the former. We start out by giving a definition of quadratic maps from $M$ to $N$ generalizing both the one via divided powers and the one due to Passi (for $R=\Z$ or, more generally, for 2-binomial rings $R$ \cite{Wf}). These quadratic maps are represented by $R$-linear maps from a certain module $\PR{M}$ to $N$;
the goal of this paper is to express  $\PR{M}$ in terms of the simpler modules $\SY{M}$, $\LA{M}$ (the symmetric resp.\ exterior tensor square of $M$), and $\GA{M}$. For the latter we provide in section 2 a neat exact sequence in terms of the Frobenius twist. Next we must determine the structure of $\PR{R}$; its study gives rise to the notion of {\it quadratic derivations}\/ on $R$ which actually play an important role in commutative quadratic algebra. It turns out that they are represented by $R$-linear maps on the polynomial ideal $I_2$ of $R$, generated by the elements $r^2-r$, $r\in R$. This result provides a functorial presentation of $I_2$, and also leads to an interesting group theoretic application:  we use quadratic algebra to extend the classical notion of an $R$-group \cite{Wf} (nilpotent of class 2, up to now) over {\it 2-binomial}\/ coefficient rings $R$ to {\it arbitrary}\/  rings $R$, thus providing a notion of 2-step nilpotent, whence non-commutative module over $R$. In sections 4 and 6 we provide various natural exact sequences for $\PR{M}$, in terms of the simpler terms mentioned above; these sequences describe the kernels and cokernels of the canonical structure maps of $\PR{M}$. Finally we provide a presentation of the ideal $I_2$ in terms of a given presentation of the ring $R$.

\section{\Rq maps}
\label{sec:def}

Let $M$ and $N$ be $R$-modules and $f:M\to N$ be a map. The \textit{cross-effect} of $f$ is the map 
$d_f : M\times M \to N$ such that $d_f(x,y):=f(x+y)-f(x)-f(y)$. The \textit{cross-actions} are the maps $f_r : M\to N$ such that $f_r(x):=f(rx)-rf(x)$ for $r\in R$,  and the 
\textit{second cross-actions} are the maps $f_{[r]}$ such that $f_{[r]}(x):=f(rx)-r^2f(x)$.

\begin{defn}\label{df:quadmap}
For  $R$-modules $M$ and $N$  a map $f:M\to N$ is an
 \textit{\Rq map} if it satisfies the following two conditions:
\begin{enumerate}
 \item the cross-effect of $f$ is $R$-bilinear,
 \item the second cross-actions of $f$ are $R$-linear.
\end{enumerate}
An \Rq map is \textit{homogeneous} if its second cross-actions are 0.
\end{defn}

Examples of homogeneous $R$-quadratic maps are quadratic forms or $R$-bilinear maps $M=M_1\times M_2 \to N$.

Clearly, any $R$-linear map is $R$-quadratic. Moreover, the sum of an $R$-linear map and  a homogeneous $R$-quadratic map is $R$-quadratic. In particular, any pointed polynomial map of degree $\le 2$ between free $R$-modules is $R$-quadratic. More precisely, let $f: R^m \to R^n$ be given by $f(x_1,\ldots,x_m) = (F_1(x_1,\ldots,x_m),\ldots,F_n(x_1,\ldots,x_m))$ where $F_1,\ldots,F_n \in R[X_1,\ldots,X_m]$ are polynomials of degree $\le 2$ with trivial constant term. Then 
$f$ is $R$-quadratic. 

However, there are $R$-quadratic maps which do not decompose as a sum of an $R$-linear map and  a homogeneous $R$-quadratic map; for example, for $R=\Z$, the map $\Z\to \Z$, $n\mapsto \binom{n}{2}$. A sufficient criterion for the existence of such a decomposition is given by the following

\begin{prop}\label{pr:lin+hom} Suppose that $R$ contains an element $r$ such that $r$ and $r-1$ are invertible. Then any $R$-quadratic map $f: M \to N$ decomposes uniquely as a sum $f=f_1+f_2$ of an $R$-linear map $f_1$ and a homogenous $R$-quadratic map $f_2$.
\end{prop}  

This criterion is improved in example \ref{ex:calculP2RM}(2) below.

\begin{proof} We can take $f_1(x) = \frac{1}{r(r-1)}f_{[r]}(x)$ and $f_2(x) = \frac{1}{r(r-1)}f_{r}(x)$ which is 
homogenous $R$-quadratic by remark \ref{rem:vardef} below. Uniqueness of $f_1$ and $f_2$ follows from the fact that under the hypothesis any map which is $R$-linear and homogenous $R$-quadratic is trivial; in fact, $r^2f(x) =f(rx)=rf(x)$ implies $f(x)=0$ as $r^2-r$ is invertible.
\end{proof}

Note that the proposition applies whenever $R$ is a field different from $\F_2$. If $R=\F_2$  any $R$-quadratic map is homogenous.

Finally, we discuss the case $R=\Z$. Passi \cite{Pa68} defines a map $f: G\to A$  from a group $G$ to an abelian group $A$ to be (normalized) polynomial of degree $\le n$ if its linear extension $\hat{f}: \Z[G]\to A$ to the group ring 
$ \Z[G]$ of $G$ annihilates $1+I^{n+1}(G)$; here $I^n(G)$ is the $n$-th power of the augmentation ideal $I(G)$ of  $\Z[G]$. An inductive characterization of this property \cite{Q3} shows that $f$ is polynomial of degree $\le 2$ iff its cross-effect $d_f$ is homomorphic in each variable. If $G$ is abelian, this is equivalent to $f$ being $\Z$-quadratic since then $f(nx )= nf(x) + \binom{n}{2}d_f(x,x)$ by induction, whence $f_n(x) =\binom{n}{2}d_f(x,x)$ which is homogenous $\Z$-quadratic; this suffices by remark \ref{rem:vardef} below.
\medskip

Let us exhibit some elementary properties of $R$-quadratic maps.
First note that if $f$ is $R$-quadratic, $f(0)=0$ as $d_f(0,0)=0$. 
Next for $x,y,z$ in $M$ and $r,s$ in $R$
the first  condition in \ref{df:quadmap} can be written as 
\begin{align}
& f(x+y+z)-f(x+y)-f(y+z)-f(x+z)+f(x)+f(y)+f(z)=0,
\label{eq:defadd1}\\
&
f(rx+sy)-f(rx)-f(sy)-rsf(x+y)+rsf(x)+rsf(y)=0.\label{eq:defadd2}
\end{align}
Additivity of $f_{[r]}$ then follows from \eqref{eq:defadd2} with $s=r$, and its $R$-linearity can be written as:
\begin{equation}
f(rsx)-r^2f(sx)-sf(rx)+r^2sf(x)=0.\label{eq:defscal1}
\end{equation}

\begin{rem}\label{rem:vardef}
 Relation \eqref{eq:defscal1} can be written as
$f_s(rx)=r^2f_s(x)$, that is $f_s$ is a  homogeneous \Rq map. Thus we see that $f$ is \Rq iff its cross-effect is $R$-bilinear and its cross-actions are homogeneous \Rqq.
\end{rem}

Clearly the set $R\ti Quad(M,N)$ (resp. $R\ti HQuad(M,N)$) of the \Rq maps (resp. homogeneous \Rq maps) from $M$ to $N$ is an $R$-module, and pre- or postcomposition of an \Rq map (resp. homogeneous \Rq map) by an $R$-linear map is an \Rq (resp. homogeneous \Rqq) map.

Throughout this paper the tensor product of $R$-modules $M$ and $N$ is denoted by $M\tp N$ instead of $M\tp_R N$.

\begin{lem}\label{crosquR} For any $R$-modules $M,M',N$ one has a natural isomorphism
\begin{equation*}
R\ti Quad(M\oplus M',N)\cong R\ti Quad(M,N)\oplus R\ti Quad(M',N)\oplus R\ti Hom(M\tp  M',N)
\end{equation*}
\end{lem}

\begin{proof} Assume $f:M\oplus M'\to N$ is an \Rq map. Then the restriction of $f$ to
$M$ and  $M'$ yields the \Rq maps $f_1:M\to N$ and $f_2:M'\to N$. One
defines the homomorphism $h:M\tp  M'\to N$ by 
$h(x\tp x')=d_f((x,0),(0,x'))$.
Knowledge of these maps allows to uniquely reconstruct  the map $f$, because
\begin{equation*}
f(x,x')=f((x,0)+(0,x'))=f_1(x)+f_2(x')+h(x\tp x'). \qedhere
\end{equation*}\noqed
\end{proof}

\subsubsection*{Universal \Rq map}
Let $\PR{M}$ be the $R$-module generated by the elements $p(x)$, $x\in M$ satisfying the relations
\begin{eqnarray}
& p(x+y+z)-p(x+y)-p(y+z)-p(x+z)+p(x)+p(y)+p(z)=0, \label{eq:defadd1P}\\
&p(rx+sy)-p(rx)-p(sy)-rs\,p(x+y)+rs\,p(x)+rs\,p(y)=0, \label{eq:defadd2P}\\
&p(rsx)-r^2\,p(sx)-s\,p(rx)+r^2s\,p(x)=0,\label{eq:defscal1P}
\end{eqnarray} 
for $x,y,z\in M$ and $r,s\in R$. Assigning $\PR{M}$ to $M$ defines an endofunctor of the category $R$-{\bf Mod} of $R$-modules in the obvious way. Clearly,

\begin{prop}\label{prp:quadUniv}
 The map $p:M\to \PR{M}$ is universal \Rqq, that is for any $R$-module $N$  precomposition by $p$ induces a binatural isomorphism
\begin{equation*}
 R\ti Hom(\PR{M},N)\to R\ti Quad(M,N).
\end{equation*}
\end{prop}
In particular, the identity map of $M$ induces a natural  $R$-linear surjection 
\begin{equation}\label{epsilondef}
 \varepsilon: \PR{M} \twoheadrightarrow M
\end{equation}
which for $M=R$ may be regarded as kind of an augmentation, cf.\ section \ref{sec:PRR}; its kernel is determined in section \ref{sec:PRM}. 
\begin{corol}
 Let $M$ and $M'$ be $R$-modules, then 
\begin{equation*}
 \PR{M\oplus M'}\simeq \PR{M}\oplus \PR{M'}\oplus (M\tp M').
\end{equation*}
\end{corol}
This is an immediate consequence of lemma \ref{crosquR}. It means that the functor $P^2_R$ is quadratic, its cross-effect being the tensor product.
\begin{prop}\label{P2isgood}
 The functor $P^2_R$ is compatible with filtered colimits, and  for any right-exact sequence of R-modules 
$M_1\mr{f}M\mr{g}M_2\to 0$, the sequence
\begin{equation}\label{se:rexseq}
\xymatrix{
 \PR{M_1} \oplus (M_1 \tp M) \ar[rr]^{\qquad\quad (\PR{f}, w)}&& \PR{M} \ar[r]^{\PR{g}}&  \PR{M_2} \ar[r]& 0}
\end{equation}
is also exact, where 
$w(m_1\tp m) = d_p(f(m_1),m)$ for $(m_1,m)\in M_1\times M$.
\end{prop}
\begin{proof}
Suppose $M=\varinjlim\limits_i M_i$. 
By proposition \ref{prp:quadUniv} it suffices to show that for any $R$-module $N$, the map $R\ti Quad(M,N) \to \varprojlim\limits_i R\ti Quad(M_i,N)$ given by restriction from $M$ to the $M_i$'s  is bijective. Injectivity is clear. Now given a family of compatible $R$-quadratic maps $f_i: M_i \to N$ we have to prove that they can be glued together to an $R$-quadratic map from $M$ to $N$, but this is routine since \Rq maps are defined by algebraic relations.

Now consider  the exact sequence 
$M_1\mr{f}M\mr{g}M_2\to 0$. It is easy to prove that the sequence
\begin{multline*}
 \xymatrix{0\to R\ti Quad(M_2,N)\ar[r]^{\quad\ g^*}&
R\ti Quad(M,N)}\\
\xymatrix{
&\ar[rr]^{(f^*,(f\tp Id)^*d_-{})\qquad\qquad\qquad\qquad\quad}&& 
R\ti Quad(M_1,N)\times R\ti Hom(M_1\tp M,N)} 
\end{multline*}
is exact for any $N$. Thus by proposition \ref{prp:quadUniv}  sequence \eqref{se:rexseq} is also exact.
\end{proof}

In order to determine the structure of $\PR{M}$ we must first study the modules $\GA{M}$ and $\PR{R}$; this is the contents of the next two sections.\medspace

\section{Homogenous $R$-quadratic maps}
\label{sec:Gamma2}
The notion of homogenous $R$-polynomial map of degree $n$ is classical; by definition such a map admits a universal factorization through the homogenous term  $\Gamma^n_R(M)$ of the divided power algebra $\Gamma_R(M)$ on $M$, see \cite{Rob}. This plays a crucial role in the definition of strict polynomial functors \cite{FS}. We here provide an exact sequence for $\GA{M}$ which degenerates to a wellknown sequence for $R=\Z$ but seems not to appear in the literature for general rings $R$.

Recall that $\GA{M}$ is defined to be the degree 2 component of the divided power algebra $\Gamma_R(M)$. As an $R$-module it is generated by elements $\ga{x}$ and symbols $\gamma_1(x)\gamma_1(y)$ which are $R$-bilinear in $x,y\in M$, subject to the relations 
  $\ga{x+y}=\ga{x}+\ga{y}+ \gamma_1(x)\gamma_1(y)$ and $\ga{rx}=r^2\ga{x}$.

 By definition of $\GA{M}$ we have an $R$-linear homomorphism
\begin{equation}\label{eq:defw}
  w: \SY{M} \to \GA{M},\quad w(xy) = \gamma_1(x)\gamma_1(y) = \gamma_2(x+y) - \gamma_2(x) - \gamma_2(y)\,
\end{equation}
$x,y\in M$. In order to exhibit the kernel and cokernel of $w$ we need to recall the notion of {\em Frobenius twist}\/.

\begin{defn} Suppose that $2M=0$. Then the right
Frobenius twist $M^{[1]}$ of $M$ is defined to be the $R$-bimodule whose left $R$-action is the given one on $M$ but the right $R$-action is given by $ xr  =  r^2 x$ for $r\in R$, $x\in M$.

\end{defn} 

In particular, the 2-torsion subgroup ${}_2M =\{x\in M\,|\,2x=0\}$ and $M/2M$ admit a right Frobenius twist.

By construction of $\GA{M}$ it is clear that there is an isomorphism $\coker{w} \cong (R\tp_{\Z} M)/U$ sending $\gamma_2(x)$ to $\overline{1\tp x}$, $x\in M$, where $U$ is the submodule of the extended $R$-module $R\tp_{\Z} M$ generated by the elements $1\tp r x - r^2 \tp x$, $(r , x)\in R \times  M$. As $U$ contains $2r \tp x = {}- (r\tp 2x - 4r\tp x)$  we see that 
\begin{equation*}
 (R\tp_{\Z} M)/U \cong (R/2R \tp_{\Z} M)/(q\tp 1)U  
\end{equation*}
where $q: R\twoheadrightarrow R/2R$ is the canonical projection. But $U$ is the $\Z$-submodule of $R\tp_{\Z} M$ generated by the elements $s\tp r x - s r^2 \tp x$, $s,r \in R$, $x\in M$, so
\begin{equation*}
  (R/2R \tp_{\Z} M)/(q\tp 1)U \cong (R/2R)^{[1]} \tp M\,.
\end{equation*}
 Thus there is a canonical isomorphism 
\begin{equation*}
  \coker{w} \cong (R/2R)^{[1]} \tp M
\end{equation*}
sending $\overline{\gamma_2(x)}$ to $\bar{1}\tp x$, $x\in M$.

On the other hand, note that for $r \in {}_2R$ and $x\in M$ one has $w(r x^2) = 0$ as $w(x^2)=2\gamma_2(x)$. This means that the homomorphism of $R$-modules
\begin{equation} \label{eq:defd}
d:  ({}_2R)^{[1]} \tp M \to \SY{M}\,,\quad d(r\tp x) = r x^2\,  
\end{equation}
has its image in $\Ker{w}$. Summarizing the above observations we obtain an exact sequence of $R$-modules 
\begin{equation}\label{se:Gsequ} 
\SY{M}/\IM{d}  \mr{\bar{w}} \GA{M} \mr{\rho} (R/2R)^{[1]} \tp M \to 0
\end{equation}
where $\rho(\gamma_2(x)) = \bar{1}\tp x$.

\begin{lem}\label{lm:GAsequ} Sequence \eqref{se:Gsequ} is short exact if $M$ is free.
\end{lem}

\begin{proof} Let $(e_i)_{i\in I}$ be a basis of $M$. Then the elements $(e_i\tp e_j)$, $(i, j)\in I^2$ such that $i\le j$ for some total ordering of $I$, form a basis of $\SY{M}$. On the other hand, it is known that the elements $\gamma_2(e_i)$, $i\in I$, together with the elements $\gamma_1(e_i)\gamma_1(e_j)=w(e_ie_j)$, $(i, j)\in I^2$ such that $i< j$, form a basis of $\GA{M}$. Thus $\Ker{w} = \bigoplus_{i\in I} ({}_2R) \,e_i^2 \subset \IM{d}$, whence $\Ker{w}=\IM{d}$.
\end{proof}

Thus taking Dold-Puppe derived functors ${\mathcal D}_nT(-) = L_nT(-,0)$ for endofunctors $T$ of $R$-{\bf Mod} 
 we get the following terminal of a long exact homotopy sequence 
 \begin{multline*} 
{\mathcal D}_1\coker{d}(M) \to {\mathcal D}_1\GA{M} \to {\mathcal D}_1  ((R/2R)^{[1]} \tp -)(M) \\
\to {\mathcal D}_0\coker{d}(M) \to {\mathcal D}_0\GA{M} \to {\mathcal D}_0  ((R/2R)^{[1]} \tp -)(M) \to 0
 \end{multline*}
 Let $ M_1\mr{u_1} M_0 \mr{u_0} M \to 0$ be a partial free resolution of $M$. Denote by $(1,1)$ respectively $\rho_i: M_0\oplus M_0 \to M_0$ the folding map, sending $(x,y)$ to $x+y$, respectively the retraction to the $i$-th summand; and let
 $\nabla$ be  the restriction of $T((1,1))$
to the submodule $T(M_0|M_1) = \Ker{(T(\rho_1),T(\rho_2))^t: T(M_0\oplus M_0) \to T(M_0) \oplus T(M_0)}$.
Then 
\begin{equation*}
 {\mathcal D}_0T(M) = \coker\Big( T(M_1) \oplus T(M_0|M_1) \mr{\tilde{u}_1} T(M_0)\Big)
\end{equation*}
where $\tilde{u}_1 = (T(u_1) , \nabla \,T(1|u_1))$. Thus by right exactness of the tensor product we have ${\mathcal D}_0T \cong T$ for $T=
({}_2R)^{[1]} \tp -$ and $T={\rm Sym}^2_R$, hence also for $T=\coker{d}$ by the snake lemma.
Moreover, one has ${\mathcal D}_1  ((R/2R)^{[1]} \tp -) = {\rm Tor}_1^R((R/2R)^{[1]},-)$ since the functor  $T=(R/2R)^{[1]} \tp -$ is additive. We thus get the following 

\begin{thm}\label{th:GMsequ} For any $R$-module $M$ there is a natural exact sequence 
\begin{equation*}
{\rm Tor}_1^R((R/2R)^{[1]},M) \mr{\tau} 
\SY{M}/\IM{d}  \mr{\bar{w}} \GA{M} \mr{\rho} (R/2R)^{[1]} \tp M \to 0 \end{equation*}
where the connecting homomorphism $\tau$ is explicitely given as follows. 
Let $(e_i)_{i\in I}$ and $(e_j)_{j\in J}$ be basis of $M_0$ and $M_1$, resp., and let $u_1$ be represented by the matrix $(a_{ij})_{(i,j)\in I\times J}$, $a_{ij}\in R$. Let $x= \sum_j \bar{r}_j^{[1]} \tp  e_j \in \Ker{(1\tp u_1)}$, $r_j\in R$. Then for all $i\in I$ there exists $s_i\in R$ such that $\sum_j r_j a_{ij}^2 =2s_i$, and we have 
\begin{equation*} 
\tau[x] \,=\, \sum_i s_i u_0(e_i)^2 \,+\, \sum_j \sum_{i_1<i_2}r_j \,a_{i_1j} \, a_{i_2j} \,u_0(e_ {i_1}) \, u_0(e_{i_2}) +\IM{d}
\end{equation*}
\end{thm}

The explicit formula for $\tau$ is obtained by going through the snake lemma type diagram defining $\tau$. 

It is known that $w$ is injective for $R=\Z$; but $\tau$ is non trivial in general, even for principal rings: 

\begin{examps} Let $R=\Z[\sqrt{2}]$ and $M=R/\sqrt{2}R\cong \Z/2\Z$. Then $\IM{d}=0$ as ${}_2R=0$, $\SY{M} \cong M$ and $w=0$ since $w(\bar{1}) = 2\gamma_2(\bar{1}) = \gamma_2(\sqrt{2}\bar{1}) = \gamma_2(0)=0$. Hence $\tau$ is surjective and non trivial.
\end{examps}

On the other hand, it is easy to deduce from theorem \ref{th:GMsequ}  sufficient conditions forcing $\tau$ to be trivial, as follows.

\begin{corol}\label{Rprinc} Suppose that $R$ is principal, and let $M= \bigoplus_{i\in I} R/a_iR$, $a_i\in R$. Then $\tau=0$ if for any $r \in R$ and $i\in I$, $2|r a_i^2$ implies $2|r a_i$. In particular, $\tau =0$ for all $M$ if $2$ is trivial or a product of two-by-two non associated primes (no powers). 
\end{corol}

In fact, in this case we may take  $(a_{ij}) = Diag(a_i)$, whence by hypothesis, each $s_i$ in theorem \ref{th:GMsequ} is of the form $s_i=s_i'a_i$, $s_i'\in R$. Thus $\tau[x]=\sum_is_iu_0(e_i)^2 =
\sum_is_i'u_0(a_ie_i)u_0(e_i) =0$.

Note that the last condition in corollary \ref{Rprinc} is satisfied for $R=\Z$, which reproduces the wellknown fact that $w$ is injective for all $\Z$-modules $M$.

\section{Quadratic derivations and the module $\PR{R}$}
\label{sec:PRR}

Recall that the group ring $\Z[G]$ decomposes as $\Z[G] = \eta(\Z)  \oplus I(G)$ where $\eta: \Z \to \Z[G]$ is the unit map; correspondingly, the canonical injection $G\to \Z[G] $ decomposes as $g \mapsto 1 +(g-1)$, and the component  $g\mapsto g-1$ is the universal derivation on $G$. We find   similar decompositions of $\PR{R}$ and of the map $p$, leading to the notion of quadratic derivation on a ring $R$.

For $r\in R$ we denote by $p_r$ (resp $p_{[r]}$) the element 
$p(r)-r\,p(1)$ (resp. $p(r)-r^2\,p(1)$) in $\PR{R}$.

\begin{prop}\label{pr:PRRrel}
 $\PR{R}$ is generated by elements $p(1)$ and $\{p_r\}_{r\in R}$ (resp. by $p(1)$ and $\{p_{[r]}\}_{r\in R}$) subject to the relations :
\begin{align}
 p_{r+s}&=p_r+p_s+rs\,p_2&
(\text{resp. } p_{[r+s]}&=p_{[r]}+p_{[s]}+rs\,p_{[2]})\label{eq:defaddPR}\\
p_{rs}&=r\,p_s+s^2\,p_r&
(\text{resp. } p_{[rs]}&=r\,p_{[s]}+s^2\,p_{[r]}).\label{eq:defscalPR}
\end{align} 
\end{prop}
\begin{proof}
 Taking $x=y=1$ in relation \eqref{eq:defadd2P} we get  
\eqref{eq:defaddPR}. Taking $x=1$ in relation \eqref{eq:defscal1P} we get \eqref{eq:defscalPR}. Conversely, a simple computation shows that the relations \eqref{eq:defadd1P}, \eqref{eq:defadd2P} and \eqref{eq:defscal1P} are consequences of \eqref{eq:defaddPR} (or \eqref{eq:defscalPR}).
\end{proof}

\begin{rem}
 The relations \eqref{eq:defscalPR} are not symmetric in $r$  and $s$. Permuting $r$ and $s$ we get
\begin{align}
 (r^2-r)p_s&=(s^2-s)p_r&
(\text{resp. } (r^2-r)p_{[s]}&=(s^2-s)p_{[r]}).\label{eq:defcsalsym}
\end{align}

\end{rem}

\begin{corol}\label{pr:decPRR}
\begin{MYitemize}
 \item The submodule of $\PR{R}$ generated by $p(1)$ is free and 
is a direct summand of $\PR{R}$.
 \item The submodules of $\PR{R}$ generated by the elements $p_r$ and by the
elements $p_{[r]}$ 
are isomorphic. They represent the \Rq maps vanishing on 0 and 1.

 \item Any \Rq map $R\to N$ has a unique 
decomposition as sum of an $R$-linear map (resp. a homogeneous \Rq map) and an \Rq map vanishing on 0 and 1.
\end{MYitemize}
\end{corol}
\begin{proof}
 The generator $p(1)$ doesn't appear in the relations, and the relations satisfied by the elements  $p_r$ or $p_{[r]}$ are the same. 
Both decompositions are easy: $f(r)=r\,f(1)+(f(r)-r\,f(1))$ and 
$f(r)=r^2\,f(1)+(f(r)-r^2\,f(1))$.
\end{proof}

These facts lead to the following structural interpretation.

\begin{defn}\label{def:quadder} A {\it quadratic derivation}\/ on  $R$ with values in an $R$-module $M$ is a map $d: R \to M$ satisfying the relations for all $r,s\in R$
\begin{align}
 d(r+s)&=d(r)+d(s)+rs\,d(2) \label{eq:defaddqd}\\
d(rs)&=r\,d(s)+s^2\,d(r) \label{eq:defscalqd}
\end{align} 
\end{defn}

\begin{examps}\label{ex:uqd}
 \begin{enumerate}
  \item Let $\eta : R \to \PR{R}$ be the ``unit map" $\eta(r)=r\,p(1)$. Then by the foregoing, the maps $D_1,D_2: R \to \PR{R}/\eta(R)$ defined by $D_1(r) = \overline{p_r}$ and $D_2(r)=\overline{p_{[r]}}$ are both universal quadratic derivations. Moreover, the canonical map $p: R\to \PR{R} =\eta(R) \oplus \langle p_r\rangle_{r\in R}$ decomposes as $p(r)= \eta(r)+D_1(r)$. This is the precise analogue with the situation in groups mentioned at the beginning of the section. Also, a quadratic derivation is the same as an $R$-quadratic map vanishing on 0 and 1.
  
  \item Let $R$ be a 2-binomial ring, i.e. for all $r\in R$ the element $r(r-1)$ is uniquely 2-divisible so that $\binom{r}{2}=\frac{r(r-1)}{2}\in R$. Then the map $h: R\to R$, $h(r)=\binom{r}{2}$, is a quadratic derivation.
  
  \item Quadratic derivations also occur naturally in the theory of square rings, cf.\ \cite{BHP}. Let $(R\mr{H}M\mr{P}R)$ be a square ring with $P=0$ \cite[8.6]{BHP}. Then $M$ is an $R$-bimodule, and $H$ satisfies relation
\ref{eq:defaddqd} and $H(rs) =r^2 H(s) + H(s)r$ for all $r,s\in R$. So if $R$ is commutative and the right and left $R$-actions on $M$ coincide then $H$ is a quadratic derivation.
This situation actually generalizes example (2) as for a 2-binomial ring $R$ we have the square ring $R_{nil}=(R \mr{h}R \mr{0}R)$ which has an important interpretation: its modules are the nilpotent $R$-groups of class 2, see \cite[8.5]{BHP}, \cite{Wf}, and also remark \ref{rm:R-grps} below.
 \end{enumerate}
\end{examps}

The surprising result now is that quadratic derivations, unlike linear, i.e.\ classical ones, are represented by an ideal of $R$ itself: let
$I_2$ denote the ideal of $R$ generated by the elements $r^2-r$, $r\in R$.

\begin{thm}\label{th:uqd} The map $D: R \to I_2$, $D(r)=r^2-r$, is a universal quadratic derivation.
\end{thm}
\begin{proof} As $D$ is a quadratic derivation it induces an $R$-linear map $\hat{D}: \langle p_r\rangle_{r\in R} \to I_2$ such that $\hat{D}(p_r)=r^2-r$, by universality of $D_1$. As $\hat{D}$ is clearly surjective we must prove its injectivity. 
Let $x=\sum_i\lambda_ip_{r_i}$ such that $\hat{D}(x)=0$, with $\lambda_i,r_i\in R$. We then have $y=\sum_i\lambda_i(r_i^2-r_i)=0$ and $p_y=0$. And by \eqref{eq:defaddPR}
\begin{equation*}
 0=p_y=\sum_ip_{\lambda_i(r_i^2-r_i)}+
\sum_{i<j}\lambda_i(r_i^2-r_i)\lambda_j(r_j^2-r_j)p_2\,.
\end{equation*}
Using \eqref{eq:defcsalsym} twice we get
\begin{equation*}
 \sum_{i<j}\lambda_i(r_i^2-r_i)\lambda_j(r_j^2-r_j)p_2=
2\sum_{i<j}\lambda_i\lambda_j(r_j^2-r_j)p_{r_i}=
\sum_{i\neq j}\lambda_i\lambda_j(r_j^2-r_j)p_{r_i} \,.
\end{equation*}
On the other hand, using \eqref{eq:defscalPR} and \eqref{eq:defcsalsym} we get
\begin{equation*}
 \sum_ip_{\lambda_i(r_i^2-r_i)}= \sum_i\lambda_i^2p_{(r_i^2-r_i)}+ \sum_i(r_i^2-r_i)p_{\lambda_i}= \sum_i\lambda_i^2p_{(r_i^2-r_i)}+ \sum_i(\lambda_i^2-\lambda_i)p_{r_i},
\end{equation*}
and using \eqref{eq:defaddPR}, \eqref{eq:defscalPR} et \eqref{eq:defcsalsym}
\begin{equation*}
 p_{(r_i^2-r_i)}=p_{r_i^2}-p_{r_i}-r_i(r_i^2-r_i)p_2=
(r_i^2+r_i)p_{r_i}-p_{r_i}-2r_ip_{r_i}=(r_i^2-r_i-1)p_{r_i}.
\end{equation*}
Then we get
\begin{equation*}
 \sum_ip_{\lambda_i(r_i^2-r_i)}=
\sum_i\lambda_i^2(r_i^2-r_i-1)p_{r_i}+
\sum_i(\lambda_i^2-\lambda_i)p_{r_i}=
-\sum_i\lambda_ip_{r_i}+\sum_i\lambda_i^2(r_i^2-r_i)p_{r_i}
\end{equation*}
and finally
\begin{align*}
 0=p_y&=-x+\sum_i\lambda_i^2(r_i^2-r_i)p_{r_i}+
\sum_{i\neq j}\lambda_i\lambda_j(r_j^2-r_j)p_{r_i}\\
&=
-x+\sum_{i,j}\lambda_i\lambda_j(r_j^2-r_j)p_{r_i}=
-x+\sum_{i}\lambda_i\Bigl(\sum_j\lambda_j(r_j^2-r_j)\Bigr)p_{r_i}
=-x.
\end{align*}
Thus $\hat{D}$ is injective.
\end{proof}

As an interesting ring-theoretic consequence we find that the ideal $I_2$ admits the following functorial presentation as an $R$-module:

\begin{corol}\label{cor:presI2}
 The ideal $I_2$ of $R$ is generated by the elements  $\varpi_r=r^2-r$, $r\in R$, subject only to the formal relations
\begin{align*}
 \varpi_{r+s}&=\varpi_r+\varpi_s+rs\,\varpi_2&
\varpi_{rs}&=r\,\varpi_s+s^2\,\varpi_r
.
\end{align*}
for $r$ and $s$ in $R$.
\end{corol}
In section \ref{sec:presentI2} we will simplify this presentation in case $R$ itself is given by a presentation. It would be interesting to know which other polynomial ideals admit analogous functorial presentations.

Combining corollary \ref{pr:decPRR} with theorem \ref{th:uqd}  furnishes the following computation of $\PR{R}$:

\begin{corol}\label{th:I2} There is a natural $R$-linear isomorphism $\PR{R} \to R\oplus I_2$ sending $p(r)$ to $(r,r^2-r)$.
\end{corol}

In the sequel we identify $\PR{R}$ and $R\oplus I_2$; the map $p$ then reads  $p(r)=(r,r^2-r)$.
\begin{examps}
 \begin{enumerate}
  \item If $I_2=R$, in particular if there is $r\in R$ such that $r^2-r$ is invertible (for example if $R$ is a field different from $\F_2$) then $\PR{R}=R\oplus R$, with $p(r)=(r,r^2-r)$.

 \item If $R$ is a 2-binomial ring (for example if $R=\Z$)  then $I_2\simeq R$, and we again have  $\PR{R}=R\oplus R$ but $p(r)=(r,\binom r2)$.
 
\item If $I_2=0$ for example if $R=\F_2$ or $R=\F_2^n$, we then have $\PR{R}=R$ with $p(r)=r$.
 \end{enumerate}
\end{examps}

\begin{rem}\label{rm:R-grps} Based on example \ref{ex:uqd}(3) and theorem \ref{th:uqd} we can now define a notion of nilpotent $R$-group of class 2 for any (even non-2-binomial!) ring $R$, as being a module over the  square ring $R_{Nil} = (R \mr{D} I_2 \mr{0} R)$. This generalizes the classical notion of  a nilpotent $R$-group of class 2
in the 2-binomial case since then the map $\times 2: R \to I_2$, $\times2(r) = 2r$, is an $R$-linear isomorphism, whence $R_{Nil}$ is isomorphic with the square ring $R_{nil}$ in example \ref{ex:uqd}(3). For example,  taking $R=\Z/q^2\Z$, $q$ prime, $R$ is  2-binomial unless $q=2$; in this case a module  over 
$R_{Nil}$ is the same as a group $G$ whose third term $G^4 \gamma_2(G)^2\gamma_3(G)$ of the lower 2-central series of Lazard is trivial \cite[8.1]{BHP}. These groups play a  role in the unstable Adams spectral sequence, and constitute algebraic models for unstable Moore-spaces whose homology is of exponent 2 \cite[8.2]{BHP}. Note that in general, an $R$-group has not only unary operations parametrized by the elements of $R$ but also binary operations paramatrized by the elements of $I_2$; in the special case where $I_2=2R$ (in particular if $R$ is 2-binomial) the latter are all multiples of the commutator by ring elements and thus determined by the group structure and the unary operations.

This observation allows to enrich quadratic algebra so as to admit  coefficients in a fixed commutative ring $R$; in particular this leads to a notion of square algebras over $R$ in the category of which $R_{Nil}$ is the initial object. Thus one obtains
 a unified  framework for dealing with nilpotent $R$-groups of class 2 on the one hand and with algebras over a nilpotent operad of class 2 over $R$ on the other hand, among others; this is work in progress.
\end{rem}

\section{The module $\PR{M}$}
\label{sec:PRM}
In this section we describe $\PR{M}$ as an extension with cokernel $M$ whose kernel is an intricate amalgamation of the simpler modules $\SY{M}$ and $\GA{M}$ invoking also the ideal $I_2$. This amalgamation will be further analyzed in the subsequent sections.\medskip

Let $M$ be an $R$-module. The kernel of the map $\varepsilon $ defined in (\ref{epsilondef}) contains  $d_p(x,y)=p(x+y)-p(x)-p(y)$, and since $d_p$ is $R$-bilinear, we get a $R$-linear map
\begin{align*}
\varphi_1:&\:\SY{M}\to\PR{M} & \varphi_1(xy)&:=p(x+y)-p(x)-p(y).
\end{align*} 
The kernel of $\varepsilon $ also contains $p_r(x)=p(rx)-rp(x)$. By  remark \ref{rem:vardef}, the map $(r,x)\mapsto p_r(x)$ is a homogeneous \Rq map in $x$ and is an \Rq map in $r$ vanishing on 0 and 1. We thus obtain an $R$-linear map
\begin{align*}
 \varphi_2:&\:I_2\tp\GA{M}\to \PR{M} & 
\varphi_2((r^2-r)\tp \ga{x} &=p(rx)-rp(x).
\end{align*}
The maps $\varphi_1$ and $\varphi_2$, together with $\varepsilon$, are the main structure homomorphisms of $\PR{M}$ as they encode the cross effect and the cross actions of the map $p$.
Clearly $\Ker \varepsilon $ is generated by the images of $\varphi_1$ and $\varphi_2$. Thus we get the exact sequence
\begin{align*}\xymatrixcolsep{3pc}
\xymatrix{\SY{M}\oplus (I_2\tp\GA{M})
\ar[r]^{\hskip17mm(\varphi_1,\varphi_2)} 
&\PR{M}\ar[r]^{\quad \varepsilon }&M\ar[r]&0
}
\end{align*}

We now give a complete description of the kernel of $\varepsilon $.

\subsubsection*{Notations}
Consider the following $R$-linear maps
\begin{align*}
 v: \SY{M}&\to I_2\tp \SY{M}& v(xy)&=2\tp xy\\
w: \SY{M}&\to \GA{M}& w(xy)&=d_{\gamma_2}(x,y)\ (\text{cf.\eqref{eq:defw}})\\
j_{11}:I_2\tp \SY{M}&\to \SY{M}& 
j_{11}((r^2-r)\tp xy)&=(r^2-r)xy\\
j_{12}:\GA{M}&\to \SY{M}& j_{12}(\ga{x})&=x^2\\
j_{21}:I_2\tp \SY{M}&\to I_2\tp \GA{M}& j_{21}&=Id\tp w \\
j_{22}:\GA{M}&\to I_2\tp \GA{M}& j_{22}(\ga{x})&=2\tp \ga{x}.
\end{align*}

\begin{lem}\label{lm:lemA}
 These maps satisfy the relations
\begin{align*}
 j_{11}v &=j_{12}w & j_{21}v &=j_{22}w &
\varphi_1j_{11}&=\varphi_2j_{21}&
\varphi_1j_{12}&=\varphi_2j_{22}.
\end{align*}
\end{lem}
\begin{proof}
 For the third relation we get:
\begin{equation*}
 \varphi_1j_{11}((r^2-r)\tp xy)=\varphi_1((r^2-r)xy) =(r^2-r)\varphi_1(xy),
\end{equation*}
and using the relation \eqref{eq:defadd2}
\begin{align*}
 \varphi_2j_{21}((r^2-r)\tp xy) &=\varphi_2((r^2-r)\tp(\ga{x+y}-\ga{x}-\ga{y})\\
&=p(r(x+y))-p(rx)-p(ry)-r\bigl(p(x+y)-p(x)-p(y)\bigr)\\
&=d_p(rx,ry)-r\,d_p(x,y)=(r^2-r)d_p(x,y)=(r^2-r)\varphi_1(xy).
\end{align*}
The other relations are easy.
\end{proof}
Let $K'(M)$ be the pushout of the diagram
\begin{equation*}
 \xymatrix{I_2\tp \SY{M}&\SY{M}\ar[l]_{\qquad v \ }
\ar[r]^{\quad w }&\GA{M}},
\end{equation*}
with structure maps $\eta_1$ and $\eta_2$, and let $K(M)$ be the pushout of the diagram
\begin{equation}\label{po:pusho2}\xymatrixcolsep{3pc}
 \xymatrix{\SY{M}&(I_2\tp \SY{M})\oplus \GA{M}\ar[l]_{(j_{11},j_{12})\hspace{13mm}}
\ar[r]^{\hspace{12mm}(j_{21},j_{22})}&I_2\tp \GA{M}},
\end{equation}
with structure maps $\theta_1$ and $\theta_2$, see the diagram below.

\begin{corol}\label{cor:2pushouts}
 The following diagram, where $j_{12}=j_1\eta_2$ and $j_{21}=j_2\eta_1$, is commutative~:
\begin{equation}\label{dg:diagr1}
 \xymatrix{\SY{M}\ar[r]^{w }\ar[d]_{v }
&\GA{M}\ar[rd]^{j_{22}} \ar[d]_{\eta_2}\\
I_2\tp \SY{M}\ar[r]^{\qquad\eta_1}\ar[rd]_{j_{11}}
&K'(M)\ar[d]_{j_1}\ar[r]^{j_2\qquad}&I_2\tp \GA{M} \ar[ddr]^{\varphi_2}\ar[d]_{\theta_2}\\
&\SY{M}\ar[drr]_{\varphi_1}\ar[r]^{\qquad\theta_1} &K(M)\ar[dr]^\varphi\\
&&&\PR{M}}
\end{equation}  
and the two squares are pushouts.
\end{corol}
The structure of $\PR{M}$ is determined by the following:
\begin{thm}\label{th:main}
 For any $R$-module, the natural sequence of $R$-modules
\begin{equation*}
 \xymatrix{0\ar[r]&K(M)\ar[r]^\varphi&\PR{M}\ar[r]^{\quad\varepsilon} &M\ar[r]&0}
\end{equation*} 
is exact. More precisely, the set $K(M)\times M$ with the operations
\begin{align*}
 (k,x)+(k',y)&=(k+k'-\theta_1(xy),x+y)\\
r\cdot (k,x)&=(rk-\theta_2\bigl((r^2-r)\tp \ga{x}\bigr),rx)
\end{align*}
is an $R$-module and the map $p(x)\mapsto (0,x)$ defines an $R$-linear isomorphism between $\PR{M}$ and this module.
\end{thm}
\begin{proof}
 Denote by $P$ the set $K(M)\times M$ with the above defined operations. Straightforward calculations using the commutativity of diagram \eqref{dg:diagr1} show that $P$ is an $R$-module. Moreover the map $M\to P$, $x\mapsto (0,x)$ is \Rqq. We then get an $R$-linear map $\PR{M} \to P$, $p(x)\mapsto (0,x)$. Moreover, the following diagram is commutative with exact rows:
\begin{equation*}
 \xymatrix{&K(M)\ar[r]^{\varphi}\ar[d]_=&
\PR{M}\ar[r]^{\quad\varepsilon} \ar[d]&M\ar[r]\ar[d]^=&0\\
0\ar[r]&K(M)\ar[r]&P\ar[r]&M\ar[r]&0
}
\end{equation*}
Thus $\varphi$ is injective and by the five lemma $\PR{M}$ and $P$ are isomorphic.
\end{proof}
\begin{rem}
 As a consequence we get the exact sequence
\begin{multline*}
 \xymatrix{(I_2\tp \SY{M})\oplus\GA{M}
\ar[rr]^{\hskip25mm
\bigl(\begin{smallmatrix} j_{11}&j_{12}\\
-j_{21}&-j_{22}\end{smallmatrix}\bigr)}
&\quad & }\\
\xymatrix{&\SY{M}\oplus (I_2\tp\GA{M})
\ar[rr]^{\hskip18mm(\varphi_1,\,\varphi_2)}
&&\PR{M}\ar[r]^{\quad \varepsilon}&M\ar[r]&0}
\end{multline*} 
\end{rem}

\begin{examps}\label{ex:calculP2RM}
 \begin{enumerate} 
  \item Suppose $R$ be a 2-binomial ring. In the 
diagram \eqref{dg:diagr1} we get $I_2\tp\SY{M}=\SY{M}$, 
$v =Id$, $I_2\tp\GA{M}=\GA{M}$ and $j_{22}=Id$. We then get $K'(M)=\GA{M}$, $\eta_2=Id$, $\eta_1=w $, $j_1=j_{12}$, $j_2=Id$. Thus $K(M)=\SY{M}$, with $\theta_1=Id$ and $\theta_2=j_{12}$. Finally we get $\PR{M}=\SY{M}\times M$ with the operations
\begin{align*}
 (k,x)+(k',y)&=(k+k'-xy,x+y)\\
r\cdot (k,x)&=(rk-\binom{r}{2} x^2,rx)
\end{align*}
Moreover, if $M$ is an $R[1/2]$-module, the 2-cocycle $(x,y)\mapsto -xy$ is the coboundary of the map $x \mapsto x^2/2$, and we get the $R$-linear isomorphism $\PR{M}\simeq \SY{M}\oplus M$, 
$(k,x)\mapsto (k+x^2/2,x)$. Thus the map $M\to \SY{M}\oplus M$, 
$x\mapsto (x^2/2,x)$ is universal \Rqq.

 \item Suppose $I_2=R$. We then get $I_2\tp\SY{M}=\SY{M}$, 
$v =2Id$, $I_2\tp\GA{M}=\GA{M}$, $j_{22}=2Id$ and $j_{11}=Id$. Thus $\eta_1$ is injective with $j_1$ as retraction. Then $\SY{M}$ is a direct summand of $K'(M)$ and we get 
$K'(M)=\SY{M}\oplus \coker w $, with $\eta_1=(Id,0)$, $\eta_2=(j_{12},\rho)$. We then obtain $j_1=(Id,0)$ and $j_2=(w ,0)$. Hence the summand $\coker w $ does not interfer in the computation of 
$K(M)$, and we get $K(M)=\GA{M}$, $\theta_2=Id$ and $\theta_1=w $. It follows that $\PR{M}=\GA{M}\times M$ with the operations
\begin{align*}
 (k,x)+(k',y)&=(k+k'-w (xy),x+y)\\
r\cdot (k,x)&=(rk-(r^2-r)\ga{x},rx).
\end{align*}
But the 2-cocycle $(x,y)\mapsto -w (xy)$ is the coboundary of the map $x\mapsto \ga{x}$. Thus we get an $R$-linear isomorphism $\PR{M}\simeq \GA{M}\oplus M$, 
$(k,x)\mapsto (k+\ga{x},x)$, and the map $M\to \GA{M}\oplus M$, 
$x\mapsto (\ga{x},x)$ is universal \Rqq. This fact generalizes proposition \ref{pr:lin+hom}.

\item Suppose now $I_2=0$. We then get $K'(M)=\coker w =M$ (by theorem \ref{th:GMsequ}) and $j_1(x)=x^2$. Thus $K(M)=\coker j_1=\LA{M}$. We obtain $\PR{M}=\LA{M}\times M$ with the operations
\begin{align*}
 (k,x)+(k',y)&=(k+k'-x\wedge y,x+y)\\
r\cdot (k,x)&=(rk,rx).
\end{align*}
It is not difficult to see that finally $\PR{M}\simeq \GA{M}$. (This also is  an easy consequence of the exact sequence \eqref{se:segamma} below).
\end{enumerate}
\end{examps}

 \section{Kernels and cokernels of some maps related to $\PR{M}$}
\label{sec:kercoker}

This section is of purely technical nature; in order to further analyze the module $K(M)=\Ker \epsilon$ in section \ref{exsequsforP2}
we here compute the kernels and cokernels of most of the maps appearing in diagram \eqref{dg:diagr1}, at least in the case where $M$ is free.

\begin{prop}\label{pr:cokerdiag}
 In   diagram \eqref{dg:diagr1} we get the following cokernels :
\begin{subequations}
\begin{align}
\coker v \simeq \coker \eta_2&
           \simeq (I_2/2R)\tp \SY{M}\label{eq:coker1}\\
\coker w \simeq \coker \eta_1&\simeq (R/2R)^{[1]}\tp M 
             \label{eq:coker2}\\
\coker j_{11}&\simeq (R/I_2)\tp \SY{M}\label{eq:coker3}\\
\coker j_{12}&\simeq (R/2R)\tp\LA{M}\label{eq:coker4}\\ 
\coker j_{21}&
          \simeq I_2\tp (R/2R)^{[1]}\tp M 
           \simeq (I_2/2I_2)^{[1]}\tp M\label{eq:coker5}\\ 
\coker j_{22}&\simeq (I_2/2R)\tp\GA{M}\label{eq:coker6}\\ 
\coker j_1\simeq \coker \theta_2&
            \simeq (R/I_2)\tp \LA{M}\label{eq:coker7}\\ 
\coker j_2\simeq \coker \theta_1&
               \simeq (I_2/2R)^{[1]}\tp M.\label{eq:coker8}
\end{align}
\end{subequations}
\end{prop}
\begin{proof}
 Since in the diagram the squares are pushouts the cokernels of each pair of opposite maps are isomorphic. The isomorphisms \eqref{eq:coker1}, \eqref{eq:coker3}, \eqref{eq:coker4} and \eqref{eq:coker6} are easy. The isomorphisms \eqref{eq:coker2} and \eqref{eq:coker5} are consequences of \ref{lm:GAsequ}. Since $\coker j_1$ is isomorphic to the cokernel of the map $\coker \eta_1 \to \coker j_{11}$ induced by $j_1$   we get 
\begin{equation*}
 \coker j_1\simeq \coker\left(  
(R/2R)^{[1]}\tp M \to (R/I_2)\tp \SY{M}\right)
\simeq (R/I_2)\tp \LA{M}
\end{equation*}
and \eqref{eq:coker7} is proved. For \eqref{eq:coker8} we use the same argument
\begin{align*}
 \coker j_2\simeq&
\coker\left((I_2/2R)\tp\SY{M}\to (I_2/2R)\tp \GA{M}\right) \\
\simeq &(I_2/2R)\tp (R/2R)^{[1]}\tp M  
\simeq (I_2/2R)^{[1]} \tp M. \qedhere
\end{align*}\noqed
\end{proof}

\begin{prop}\label{pr:kerdiag}
 Suppose $M$ is a free $R$-module, $M=\oplus_iR$. Recall that ${}_2N$ is the 2-torsion submodule of $N$ for any $R$-module $N$. Then in  diagram \eqref{dg:diagr1} we have the following kernels :
\begin{subequations}
\begin{align}
\Ker v &=\oplus_{i'\leq i''}\ {}_2R={}_2\SY{M}&
\Ker w &=\oplus_i\ {}_2R \simeq {}_2R^{[1]}\tp M\label{eq:ker1}\\
\Ker j_{11}&=0&
\Ker j_{12}&=\oplus_{i'< i''}\ {}_2R \simeq {}_2\LA{M}\label{eq:ker2}\\
\Ker j_{21}&=\oplus_i\ {}_2I_2 \simeq {}_2I_2^{[1]}\tp M&
\Ker j_{22}&=\oplus_{i'\leq i''}\ {}_2R={}_2\GA{M}\label{eq:ker3}\\
\Ker \eta_1&=0&
\Ker \eta_2&=\oplus_{i'< i''}\ {}_2R \simeq {}_2\LA{M}\label{eq:ker4}
\end{align}\begin{align}
\Ker j_1&=\oplus_i\ I_2/2R \simeq (I_2/2R)^{[1]}\tp M&
\Ker j_2&=\oplus_i\ ({}_2R\oplus I_2/2R)\label{eq:ker5}\\
&&&\simeq ({}_2R\oplus I_2/2R)^{[1]}\tp M\notag \\
\Ker \theta_1&=\oplus_i\ {}_2R \simeq {}_2R^{[1]}\tp M&
\Ker \theta_2&=0.\label{eq:ker6}
\end{align}
\end{subequations}
\end{prop}
\begin{proof}
 Suppose first that $M=R$. Then diagram \eqref{dg:diagr1} becomes
\begin{equation}\label{dg:diagr2}
 \xymatrix{R\ar[r]^{w }\ar[d]_{v }
&R\ar[rd]^{j_{22}} \ar[d]_{\eta_2}\\
I_2\ar[r]^{\eta_1}\ar[rd]_{j_{11}}
&K'(R)\ar[d]_{j_1}\ar[r]^{j_2}&I_2 \ar[d]_{\theta_2}\\
&R\ar[r]^{\theta_1} &K(R)}
\end{equation}  
with the maps
\begin{align*}
 v (x)&=2x,&w &=2Id,&j_{11}(x)&=x,&
j_{12}&=Id,&j_{21}&=2Id,&j_{22}(x)&=2x.
\end{align*}
We then get the following kernels:
\begin{align*}
 \Ker v =\Ker w &=\Ker j_{22}={}_2R,&
\Ker j_{21}&={}_2I_2,\\
\Ker j_{12}&=\Ker j_{11}=0,&
\Ker\eta_1&=\Ker\eta_2=0.
\end{align*}
Since $j_{12}=Id$, $R$ is a summand of $K'(R)$, and $K'(R)=R\oplus \Ker j_1$. Now $\Ker j_1=\{\eta_1(x) - \eta_2(x)\,|\,x\in I_2\}$, hence $j_2 \Ker j_1=0$. Thus  $ j_2(x,y)=2x$ for $(x,y)\in R \times \Ker j_1$, whence
\begin{align*}
 \Ker j_2&={}_2R\oplus \Ker j_1,&K(R)=I_2, & &\theta_2=Id, & &\theta_1(x)=2x\,.
\end{align*}
  So finally  $\Ker \theta_1={}_2R$ and $\Ker \theta_2=0$.

Now let $M$ be a free $R$-module with basis $\{e_i\}_{i\in I}$. The module $\SY{M}$ (resp. $\GA{M}$) is also free with basis $\{e_i^2\}_{i\in I}\cup \{e_{i'}e_{i''}\}_{i'<i''}$ (resp. $\{\ga{e_i}\}_{i\in I}\cup \{\gamma_1(e_{i'})\gamma_1(e_{i''})\}_{i'<i''}$). Since any map in  diagram \eqref{dg:diagr1} acts diagonally with respect to these bases, it is sufficient to consider the effect of the maps on one square term, that is the case $M=R$ above, and the effect of the maps on one rectangular term, that is $e_{i'}e_{i''}$ (resp. $\gamma_1(e_{i'})\gamma_1(e_{i''})$). In the latter case we have the same diagram as in \eqref{dg:diagr2}, but the maps are:
\begin{align*}
v (x)&=2x&
w &=Id&
j_{11}(x)&=x& 
j_{12}&=2Id&
j_{21}&=Id&
j_{22}(x)&=2x
\end{align*}
We then obtain $K'=I_2$, $\eta_1=Id$, $\eta_2=v $, $j_2=Id$, $K=R$, $\theta_1=Id$ and $j_1=\theta_2=j_{11}$. Thus for a rectangular term we get 
\begin{equation*}
 \Ker v =\Ker j_{12}=\Ker j_{22}=\Ker \eta_2={}_2R,
\end{equation*}
and the other kernels are zero. Summarizing the results above we obtain the proposition.
\end{proof}
\begin{rem} As a byproduct of the proof we get that for a free $R$-module $M$:
 \begin{align*}
K'(M)&=\left(\oplus_i(R\oplus(I_2/2R))\right)\oplus
\left(\oplus_{i'<i''}I_2\right)\\
K(M)&=\left(\oplus_i\ I_2\right)\oplus
\left(\oplus_{i'<i''}R\right).
 \end{align*}
\end{rem}

\section{Exact sequences for $\PR{M}$}
\label{exsequsforP2}

We are now ready to compute the 
kernels and cokernels of the structure maps $\varphi_1$ and $\varphi_2$, thus providing natural exact sequences expressing $\PR{M}$ in terms of the ideal $I_2$ and the simpler functors ${\rm Sym}^2_R$, $\Lambda^2_R$ and $\Gamma^2_R$.

\subsection{The map $\varphi_1:\SY{M}\to \PR{M}$}
\label{sec:mapphi1}
\begin{lem}
 Let $M$ be an $R$-module and $d$ be the map defined in \ref{eq:defd}. Then the sequence
\begin{equation}\label{se:shexseqphi1}
 \xymatrix{\SY{M}/\IM{d}\ar[r]^{\qquad\ \bar\varphi_1}&\PR{M}\ar[r]^{q_1} &\coker \varphi_1\ar[r]&0}
\end{equation}
is exact. It is short exact if $M$ is free.
\end{lem}
\begin{proof}
For $r\tp x$ in $({}_2R)^{[1]}\tp M$ we have using \eqref{eq:defscal1P} for $r=2$ and $s=r$
\begin{equation*}
 \varphi_1(d(r\tp x))=\varphi_1(rx^2)=r\varphi_1(x^2)=
r(p(2x)-2p(x))=rp(2x)=-4p(rx)=0,
\end{equation*}
hence $\IM{d}\subset \Ker \varphi_1$, the map $\bar\varphi_1$ is defined and the first part is proved. Suppose moreover $M$ is free. Since $\varphi$ is injective, $\Ker \varphi_1=\Ker \theta_1=\IM{d}$ by \eqref{eq:ker6}.
\end{proof}

\begin{thm}
Let $\psi_1\ :(I_2/2R)^{[1]}\tp M\to \coker \varphi_1$ be the map defined by
$\psi_1(\overline{r^2-r}\tp x) =  q_1\varphi_2((r^2-r)\tp \gamma_2(x))$ and let $\varepsilon_1\ :\coker\varphi_1\to M$ be the map induced by $\varepsilon$.  We have the following two natural exact sequences:
\begin{equation}\label{se:cokphi1}
 \xymatrix{0\ar[r]&(I_2/2R)^{[1]}\tp M \ar[r]^{\qquad\psi_1}& 
\coker \varphi_1 \ar[r]^{\quad\varepsilon_1} &M\ar[r]&0,}
\end{equation}
\begin{equation*}
 \xymatrix{{\rm Tor}\/_1^R((I_2/2R)^{[1]},M)\ar[r]^{\ \tau_1} &  \SY{M}/\IM{d}\ar[r]^{\qquad\bar\varphi_1}& \PR{M}\ar[r]^{q_1}&
\coker (\varphi_1)\ar[r]& 0 }
\end{equation*}
%
with $\tau_1=\tau\,\rond \, {\rm Tor}\/_1^R(\iota^{[1]}, Id)$, where $\tau$ is defined in Theorem \ref{th:GMsequ} and $\iota:I_2/2R\to R/2R$ is  the inclusion.
\end{thm}

Taking $R=\Z$ we rediscover the exact sequence $0\to {\rm Sym}^2_{\Z}(M) \to P^2_{\Z}(M) \to M \to 0$ due to Passi \cite{PaQ2}. 

\begin{proof}
 Since $K(M)$ is the kernel of $\varepsilon$ and $\varepsilon=\varepsilon_1q_1$  we get $\coker\theta_1=\Ker\varepsilon_1$, so by \eqref{eq:coker8} the first sequence is exact. Now since  the exact sequence \eqref{se:shexseqphi1} is short exact when  $M$ is free, we can left-complete it by the first derived functor $\mathcal D^1(\coker \varphi_1)$ with connecting morphism $\tau'$.
But applying the long exact homotopy sequence  to the sequence \eqref{se:cokphi1} we obtain
\begin{equation*}
 0=\mathcal D^2(Id)\to \mathcal D^1 ((I_2/2R)^{[1]}\tp -)
\mr{\mathcal D^1(\psi_1)}\mathcal D^1(\coker\varphi_1) \to\mathcal D^1(Id)=0
\end{equation*}
hence $\mathcal D^1(\psi_1)$ is an isomorphism, so $\tau_1=\tau'\,\rond \,\mathcal D^1(\psi_1)$. Now consider the diagram
\begin{equation*}
 \xymatrix{
\mathcal D^1(\coker\varphi_1)\ar[d]^{\mathcal D^1(\bar g_2)}
\ar[r]^{\tau'} &  \SY{M}/\IM{d}\ar[r]^{\qquad\bar\varphi_1}\ar[d]^=& \PR{M}\ar[r]^{q_1}\ar[d]^{g_2}&\coker (\varphi_1)\ar[r]\ar[d]^{\bar g_2}& 0
\\
{\rm Tor}_1^R((R/2R)^{[1]},M) \ar[r]^{\ \tau}& \SY{M}/\IM{d}  
\ar[r]^{\qquad\bar{w}}& \GA{M} \ar[r]^{\rho\quad} &(R/2R)^{[1]} \tp M \ar[r]& 0
}
\end{equation*}
Its lines are exact and the central square commutes, thus the diagram is commutative, and 
$\tau_1=\tau\,\rond \,\mathcal D^1(\bar g_2)\,\rond \,\mathcal D^1(\psi_1)$. This implies the assertion since a simple computation shows that 
$\bar g_2\,\rond \,\psi_1=\iota^{[1]}\tp Id$.
\end{proof}
\vspace{-1mm}

\subsection{The map $\varphi_2:I_2\tp \GA{M}\to \PR{M}$}
\label{sec:mapphi2}

\begin{thm}
Let $\psi_2\ :(R/I_2)\tp \LA{M}\to \coker \varphi_2$ be the map defined by
$\psi_2(\overline{r}\tp x\wedge y) =  q_2\varphi_1(r\,xy )$ and $\varepsilon_2\ :\coker\varphi_2\to M$ be the map induced by $\varepsilon$.  We have the following two natural exact sequences:
\begin{equation*}
 \xymatrix{0\ar[r]&(R/I_2)\tp \LA{M} \ar[r]^{\qquad\psi_2}& 
\coker \varphi_2 \ar[r]^{\quad\varepsilon_2} &M\ar[r]&0,}
\end{equation*}
\begin{equation*}
 \xymatrix{{\rm Tor}\/_1^R((R/I_2),\SY{M})\ar[r]^{\qquad\quad \tau_2} &  I_2\tp\GA{M}\ar[r]^{\quad\varphi_2}& \PR{M}\ar[r]^{q_2}&
\coker (\varphi_2)\ar[r]& 0 }
\end{equation*}
where $\tau_2$ is the composite of the connecting morphism 
${\rm Tor}\/_1^R((R/I_2),\SY{M})\to 
I_2\tp \SY{M}
$ and of the morphism 
$j_{21}:I_2\tp \SY{M}\to I_2\tp \GA{M}$.
\end{thm}

\begin{proof}
By definition \eqref{po:pusho2} of $K(M)$  we have the pushout
\begin{equation*}\xymatrixcolsep{3pc}
 \xymatrix{(I_2\tp \SY{M})\oplus \GA{M}\ar[d]_{(j_{11},j_{12})}
\ar[r]^{\hspace{12mm}(j_{21},j_{22})}&I_2\tp \GA{M}\ar[d]^{\theta_2}\\
\SY{M}\ar[r]^{\theta_1}&K(M)}
\end{equation*}
It follows that 
\begin{equation*}
\Ker \theta_2 =  (j_{21},j_{22}) \Ker (j_{11},j_{12})\,.
\end{equation*}
One has the exact sequence
\begin{equation*}
 \xymatrix{0\ar[r]&\Ker j_{11}\ar[r]&\Ker (j_{11},j_{12})\ar[r]&\Ker \bar j_{12}\ar[r]&0}
\end{equation*}
where $\bar j_{12}$ is the composite map 
\begin{equation*}
 \xymatrix{\GA{M}\ar[r]^{j_{12}}&\SY{M}\ar[r]&(R/I_2)\tp\SY{M}} 
\end{equation*}
and where the first map is induced by the inclusion and the second one by the projection to the second factor.  Note that $w $ takes values in $\Ker \bar j_{12}$ since $j_{12}w  = 2Id$ and $2\in I_2$.
>From the commutative diagram
\begin{equation*}
 \xymatrix{&&\SY{M}\ar[r]^=\ar[d]^{(v ,-w )}&\SY{M}\ar[d]^{w }\\
0\ar[r]&\Ker j_{11}\ar[r]&\Ker (j_{11},j_{12})\ar[r]&\Ker \bar j_{12}\ar[r]&0}
\end{equation*}
we deduce the exact sequence of cokernels :
\begin{equation*}
 \xymatrix{\Ker j_{11}\ar[r]&\Ker(j_{11},j_{12})/\IM{(v ,-w )}\ar[r]& 
\Ker \bar j_{12}/\IM{ w }\ar[r]&0}
\end{equation*}
where $\Ker \bar j_{12}/\IM{w }$ can be identified with the kernel of the map 
\begin{align*}j'_{12}\ :\ 
(R/2R)^{[1]}\tp M&\to(R/I_2)\tp \SY{M}& \bar r\tp x &\mapsto \bar r\tp x^2.
\end{align*}
Clearly this kernel contains the image of $(I_2/2R)^{[1]}\tp M$. Now this map $(I_2/2R)^{[1]}\tp M\to \Ker j'_{12}$ lifts to an $R$-linear map 
\begin{align*}
 \zeta\ :\ (I_2/2R)^{[1]}\tp M&\to \Ker (j_{11},j_{12})/\IM{(v ,-w )}\\
\overline{(r^2-r)}\tp x&\mapsto \overline{((r^2-r)\ga{x},-(r^2-r)\tp x^2)}.
\end{align*}
We then get the commutative diagram
\begin{equation*}
 \xymatrix{
&(I_2/2R)^{[1]}\tp M\ar[r]^=\ar[d]^\zeta& (I_2/2R)^{[1]}\tp M\ar[d]\\
\Ker j_{11}\ar[r]&\Ker (j_{11},j_{12})/\IM{(v ,-w )}\ar[r]&
\Ker j'_{12}\ar[r]&0}
\end{equation*}
which leads to the exact sequence of cokernels : 
\begin{equation*}
 \xymatrix{\Ker j_{11}\ar[r]&
\Bigl(\Ker (j_{11},j_{12})/\IM{(v ,-w )}\Bigr)/\IM{\zeta}\ar[r]&
\Ker j''_{12}\ar[r]&0}
\end{equation*}
where $j''_{12}$ is the map $(R/I_2)\tp M\to (R/I_2)\tp \SY{M}$ such that $j''_{12}(\bar r\tp x)=\bar r\tp x^2$. Now, by the following lemma \ref{lm:car2sym}, $j''_{12}$ is injective. We then obtain a surjection 
\begin{equation}\label{eq:surj}
 \Ker j_{11}\twoheadrightarrow
\Bigl(\Ker (j_{11},j_{12})/\IM{(v ,-w )}\Bigr)/\IM{\zeta}.
\end{equation} 
On the other hand $\IM{(v ,-w )}$ is contained in $\Ker (j_{21},j_{22})$, thus the surjection $\Ker(j_{11},j_{12})$ onto 
$\Ker\theta_2$ induced by $(j_{21},j_{22})$ factors by a surjection of
 $\Ker(j_{11},j_{12})/ \IM{(v ,-w )}$ onto $ \Ker\theta_2$. But 
$\IM{\zeta}$ is also annihilated by $(j_{21},j_{22})$, so we get a surjection
\begin{equation*}
 \Bigl(\Ker (j_{11},j_{12})/\IM{(v ,-w )}\Bigr)/\IM{\zeta} \twoheadrightarrow \Ker \theta_2.
\end{equation*}
 By composition of this surjection with the one of   equation \eqref{eq:surj} we obtain a surjection 
$\Ker j_{11}\twoheadrightarrow\Ker \theta_2 = \Ker \varphi_2$ given by restriction of $j_{21}$.

We then can  conclude the proof, since, tensoring the exact sequence $I_2\rightarrowtail R\twoheadrightarrow R/I_2$ by $\SY{M}$ we get an isomorphism ${\rm Tor}\/^R_1(R/I_2,\SY{M})\to \Ker j_{11}$.
\end{proof}

\begin{lem}\label{lm:car2sym}
Let $R$ be a ring where the ideal $I_2$ is zero. Then for any $R$-module $M$ the $R$-linear map $M\to \SY{M}$, $x\mapsto x^2$ is injective.
\end{lem}
First we prove
\begin{lem}\label{lm:rigI2null}
Let $R$ be a ring, finitely generated over $\Z$, where the ideal $I_2$ is zero. Then $R$ is isomorphic to $\F_2^m$ for some integer $m$.
\end{lem}
\begin{proof}
 Let $R$ be  such a ring. Clearly the map 
\begin{align*}
 R[X]/<X^2-X>&\to R\times R & \overline{aX+b}&\mapsto (a+b,b)
\end{align*}
is a ring isomorphism. Thus the ring $\Z[X_1,\dots ,X_n]/I_2$ is  isomorphic to $\F_2^{2^n}$, and any quotient of this ring is itself isomorphic to $\F_2^m$ for some $m$.
\end{proof}
\begin{proof}[Proof of lemma \ref{lm:car2sym}]
 Suppose first $R$ is of finite type over $\Z$. By the lemma above, $R$ is a 
finite product of copies of $\F_2$, and $M$ is a finite product of $\F_2$-vector spaces. We   then must only  prove that the map $M\to {\rm Sym}^2_{\F_2}(M)$ is injective when $M$ is an $\F_2$-vector space, which is clear.

In the general case, since $M$ is a filtered direct limit of finitely presented modules, we can suppose that $M$ is of finite presentation over $R$. So $M$ is the quotient of an $R^m$ by a finite number of relations $\rho_j$, and
in these relations we have only a finite number of coefficients in $R$. Let 
$x=\sum_ir_i\overline{e_i}\in M$ such that $x^2=0$, with $\{e_i\}_i$ being the canonical basis of $R^m$. We can then write this equality over a finitely generated   subring of $R$, generated by the elements $r_i$, by the coefficients of the relations $\rho_j$, 
and by the coefficients $r_{ij}$ of the $R$-linear combination $\sum r_{ij}e_i \rho_j$ in $\SY{R^m}$ trivializing $x^2$ in $\SY{M}$.
 We  then can conclude that $x=0$.
\end{proof}

\subsection{The map $I_2\tp M\to \PR{M}$}
\label{sec:lastmap}
The canonical map $\gamma_2:M\to\GA{M}$ is \Rqq, so it factors through $\PR{M}$. We thus obtain  a surjective $R$-linear map $g_2:\PR{M}\to \GA{M}$, $g_2(p(m))=\ga{m}$. On the other hand, the map $R\times M\to \PR{M}$, $(r,m)\mapsto p_{[r]}(m)=p(rm)-r^2p(m)$ is $R$-linear in $m$, and is \Rq in $r$ and vanishes if $r=0$ or 1; whence it factors through an $R$-linear map $\chi:I_2\tp M\to \PR{M}$,  $\chi((r^2-r)\tp m):=p(rm)-r^2p(m)$. Clearly we get
\begin{prop}
The sequence
\begin{equation}\label{se:segamma}
\xymatrix{I_2\tp M\ar[r]^\chi&\PR{M}\ar[r]^{g_2}&\GA{M}\ar[r]&0} 
\end{equation}
 is exact.
\end{prop}
We are not able to compute the kernel of $\chi$ so far; this would be an easy consequence of a computation of the first derived functor of $\Gamma^2_R$ which doesn't seem to be known. So we content ourselves of two easy remarks: if $m\in{}_2M$ then $2\tp m\in \Ker \chi$, and if $A$ is the image of ${\rm Tor}\/^R_1(R/I_2,M)$ in $I_2\tp M$ by the connecting homomorphism of the exact sequence $I_2\rightarrowtail R\twoheadrightarrow R/I_2$, then $\Ker \chi\subset A$ (use the map $\epsilon$). In particular, if $I_2=2R$, we get $\Ker \chi = A = {\rm Im}(2R\tp {}_2M\to 2R\tp M) \simeq {}_2M/{}_2RM$.

\section{Generators and relations for $I_2$}
\label{sec:presentI2}
As the ideal $I_2$ plays a key role in all our results, in particular as a factor in torsion products, it is convenient to dispose of a more economic presentation of $I_2$ than the functorial one in corollary \ref{cor:presI2}. This is provided here in case  $R$ itself is given by a presentation as a quotient of a polynomial ring. We start by the following immediate calculation where we write $\varpi(x)=\varpi_x$.


\begin{lem}
 For any monomial $M=\prod_{k=1\dots n}x_k^{m_k}$, $x_k\in R$ we get
\begin{equation*}
 D(M)=M^2-M=\sum_{k=1}^n\left(x_1^{2m_1}\dots x_{k-1}^{2m_{k-1}} \left(\sum_{j=m_k-1}^{2m_k-2}x_k^j\right)
 x_{k+1}^{m_{k+1}}\dots x_n^{m_n}\right)\varpi(x_k),
\end{equation*} 
and for $P=\sum_{k=1}^n \bar a_k M_k$, with $a_k\in \Z$ and the $M_k$'s unitary monomials in the elements $x_i$, we get
\begin{equation*}
 D(P)=P^2-P=\sum_{k=1}^n \bar a_k D(M_k) + \sum_{k=1}^n\overline{\binom{a_k}{2}}M_k^2\varpi(2) + 
\sum_{1\leq k'<k''\leq n} \bar a_{k'}\bar a_{k''}M_{k'}M_{k''}\varpi(2).
\end{equation*} 
\end{lem}
In particular for a polynomial ring $\Z[X_i]_{i\in I}$, $D(P)$ is a $\Z$-linear combination of the elements $\varpi(X_i)$ and $\varpi(2)$.

Now let
$S=\Z[X_i]_{i\in I}$, 
$\mathfrak a=<P_{\alpha}(\underline{X})>_{\alpha\in A}$ and 
$R=S/\mathfrak a$.
 Let $I_*=I\uplus \{*\}$ and $X_*:=2$, and for $i\in I_*$, denote by $x_i$ the class of $X_i$ in $R$, and $\pi_i=\varpi(x_i)=x_i^2-x_i$. Then the desired presentation of $I_2$ is given by the following

\begin{prop}\label{prop:presred}
 The ideal $I_2$ of the ring $R=\Z[x_i]$, is generated by elements $\pi_i$, $i\in I_*$, subject to the relations
\begin{align}
 (x_i^2-x_i)\pi_j&=(x_j^2-x_j)\pi_i & (i,j)&\in I_*^2,& i<j \text{\quad for some total ordering.}\\
\overline{D_S(Q_\alpha)}&=0 & \alpha&\in A
\end{align} 
where $\overline{D_S(Q_\alpha)}$ is the image of $D_S(Q_\alpha)$ by the canonical map $I_2(S)\to I_2(R)$ sending $\varpi_S(X_i)$ (resp. $\varpi_S(2)$) to $\varpi_R(x_i)=\pi_i$ (resp. 
$\varpi_R(\bar 2)=\pi_*=\bar 2^2-\bar 2=\bar 2$) .
\end{prop}

 The proof requires some more notation. Let $R^*:=R-\{0\}$, $R^{**}:=R-\{0,1\}$, $J'(R):=(R^*)^2$, $J''(R):=(R^{**})^2$ and $J(R):=J'(R)\amalg J''(R)$. 
We denote by $\{[x]\}$ the canonical basis of $R^{(R^{**})}$ and by $\{[x,y]_1\}$ and $\{[x,y]_2\}$ the basis of $R^{(J'(R))}$ and $R^{(J''(R))}$, and we consider  the elements
\begin{align*}
\rho_1(x,y)&:=[x+y]-[x]-[y]-xy[2],\\
\rho_2(x,y)&:=  [xy]-x[y]-y^2[x],
\end{align*}
in $R^{(R^{**})}$ with $[0]=[1]:=0$.

\begin{lem}\label{relrho}
 We have the following relations:
\begin{subequations}
\begin{align}
\rho_1(x+y,z)&=\rho_1(x,y+z)-\rho_1(x,y)+ \rho_1(y,z),\label{relrho1}\\
\rho_1(x,y+z)&=\rho_1(y,x+z)+\rho_1(x,z) -\rho_1(y,z),\label{relrho1a}\\
\rho_1(\sum_{i=1}^n  x_i,y)&= \sum_{i=1}^n\rho_1(x_i,y+\sum_{j=1}^{i-1}x_j) -\sum_{i=2}^n\rho_1(x_i,\sum_{j=1}^{i-1}x_j), \label{relrho1s}\\
\rho_2(x+y,z)&=\rho_2(x,z)+\rho_2(y,z)+\rho_1(xz,yz) -z^2\rho_1(x,y),\label{relrho2g}\\
\rho_2(\sum_{i=1}^n  x_i,y)&=\sum_{i=1}^n \rho_2(x_i,y) \label{relrho2gs}
+\sum_{i=1}^{n-1}(\rho_1(\sum_{j=1}^i x_jy,x_{i+1}y) -y^2\rho_1(\sum_{j=1}^i x_j,x_{i+1}) ),
\\
\rho_2(x,y+z)&=\rho_2(x,y)+\rho_2(x,z) +yz(\rho_2(x,2)-\rho_2(2,x))\label{relrho2d}\\
&\hskip4cm +\rho_1(xy,xz) -x\rho_1(y,z),\notag\\
\rho_2(x,\sum_{i=1}^n y_i)&=\sum_{i=1}^n\rho_2(x,y_i)  +(\sum_{1\leq i <j\leq n}y_iy_j) (\rho_2(x,2)-\rho_2(2,x))
\label{relrho2ds}\\
&\hskip2cm +\sum_{i=1}^{n-1}(\rho_1(x\sum_{j=1}^i y_j,xy_{j+1})- x\rho_1(\sum_{j=1}^i y_j,y_{j+1})), \notag\\
\rho_2(xy,z)&=\rho_2(x,yz)+x\rho_2(y,z) -z^2\rho_2(x,y),\label{relrho3}\\
\rho_2(x,yz)&=\rho_2(y,xz) +z^2(\rho_2(x,y)-\rho_2(y,x))\label{relrho3a}
+y\rho_2(x,z)-x\rho_2(y,z),\\
\rho_2(\prod_{i=1}^n x_i,y)&=
\sum_{i=1}^n\Bigr(\prod_{j=1}^{i-1}x_j\Bigr)
\rho_2(x_i,\prod_{j=i+1}^{n}x_jy)\label{relrho3s}
-y^2 \sum_{i=1}^{n-1}\Bigr(\prod_{j=1}^{i-1}x_j\Bigr)
\rho_2(x_i,\prod_{j=i+1}^{n}x_j).
\end{align} 
\end{subequations}
\end{lem}

\begin{proof}
 By simple computation for the relations \eqref{relrho1}, \eqref{relrho2g}, \eqref{relrho2d} and \eqref{relrho3}.
Using \eqref{relrho1} to compute $\rho_1(x+y,z)$ and $\rho_1(y+x,z)$ wet get \eqref{relrho1a}. Using \eqref{relrho3} to compute $\rho_2(xy,z)$ and $\rho_2(yx,z)$ wet get \eqref{relrho3a}. Relations \eqref{relrho1s}, \eqref{relrho2gs}, \eqref{relrho2ds} and \eqref{relrho3s} are obtained by induction respectively from the relations \eqref{relrho1}, \eqref{relrho2g}, \eqref{relrho2d} and \eqref{relrho3}.
\end{proof}

\begin{proof}[Proof of proposition \ref{prop:presred}]
 By corollary \ref{cor:presI2} we have the exact sequence:
\begin{equation}\label{presP2RR}
\begin{CD}
 R^{(J(R))}@>t>>R^{(R^{**})}@>\varpi>>I_2@>>>0
\end{CD}
\end{equation} 
where $\varpi([x]):=\varpi_x=x^2-x$ for $x\in R^{**}$, $t([x,y]_1):=\rho_1(x,y)$ for $(x,y)\in J'(R)$ and $t([x,y]_2):=\rho_2(x,y)$ for $(x,y)\in J''(R)$.
Let  $K:=\{(i,i')\mid i<i'\in I^*\}$ and $J_1(R):=J(R)\amalg K$. We can extend the map $t$ to a map 
$t_1:R^{(J_1(R))} \mapsto R^{(R^{**})}$ such that 
$\IM t=\IM t_1$, by putting
\begin{align*}
 t_1((i,i'))&:=\rho_2(x_i,x_{i'})-\rho_2(x_{i'},x_i) &\text{for }(i,i')&\in K.
\end{align*}
Clearly $t_1((i,i'))$ is in $\IM t$.  

Obviously we can now replace the exact sequence \ref{presP2RR} by the following:
\begin{equation*}
 \begin{CD}
 R^{(J_2)}@>t_2>>R^{(S^{**})}@>\varpi>>I_2@>>>0
\end{CD}
\end{equation*}
with $J_2:=J_1(S)\amalg A_1$ where $A_1:=\{(x,y)\in (S^{**})^2\mid x\equiv y \text{ mod } \mathfrak a\}$.
The map $t_2$ is defined by 
\begin{align*}
 t_2((x,y))&=[x]-[y]& \text{ for }(x,y)&\in A_1,
\end{align*}
and by the composition of $t_{1}$ and the canonical map
$S^{(S^{**})}\to R^{(S^{**})}$ on $J_1(S)$.

First we will reduce step by step the set $A_1$.
\begin{MYitemize}
 \item If $x\in S$ and $a\in \mathfrak a$ then $(x+a,x)\in A_1$. We have $t_2((x+a,x))=t_2(x,0)+\rho_1(x,a)$. Without changing the image of $t_2$ we can then replace $A_1$ by $A_2=\mathfrak a\times\{0\}\simeq \mathfrak a$.
 \item By the relations $\rho_1(x,y)$ we can suppose $a$ to be a multiple of one of the polynomials $P_\alpha$.
 \item By the relations $\rho_2(x,y)$ we then can suppose $a$ to be one of the polynomials $P_\alpha$.
\end{MYitemize}

Taking $J_3:=J'(S)\amalg J''(S)\amalg K_S\amalg A$ we obtain the exact sequence
\begin{equation*}
 \begin{CD}
 R^{(J_3)}@>t_3>>R^{(S^{**})}@>\varpi>>I_2@>>>0
\end{CD}
\end{equation*}
with $t_3(\alpha)=[P_\alpha]$.

We will now reduce step by step the sets $J'(S)$ and $J''(S)$.
\begin{MYitemize}
 \item By relation \eqref{relrho2gs} it suffices to take those elements $(x,y)\in J''(S)$ where $x$ is a unitary monomial.
 \item By  relation \eqref{relrho3s} it suffices to take those elements  $(x,y)\in J''(S)$ where $x$ is a generator.
 \item By the relation \eqref{relrho2ds} it suffices to take those elements  $(x,y)\in J''(S)$ where $x$ is a generator and $y$ is a unitary monomial.
\item By relation \eqref{relrho1s} it suffices to take those elements  $(x,y)\in J'(S)$ where $x$ is a unitary monomial.
\end{MYitemize}
We can order all the unitary monomials by total degree and by lexicographic order.
\begin{MYitemize}
 \item By relation \eqref{relrho1a} it suffices to take those elements  $(x,y)\in J'(S)$ where $x$ is a unitary monomial greater or equal to any monomial in $y$.
 \item By  relation \eqref{relrho3a} it suffices to take those elements  $(x,y)\in J''(S)$ where $x$ 
  is a variable greater or equal to any variable in the unitary monomial $y$.
\end{MYitemize}

Denote by $J'_4$ the  set of elements $(x,y)$ of $J'(S)$ where $x$ is a unitary monomial greater or equal to any monomial in $y$, and by $J''_4$ the set of elements $(x,y)$ of $J''(S)$ where $y$ is a unitary monomial and $x$ is a  variable greater or equal to any variable in $y$. Let $J_4:=J'_4\amalg J''_4\amalg K\amalg A$  and let $t_4$ be the restriction of $t_3$. We then get the exact sequence
\begin{equation*}
\begin{CD}
 R^{(J_4)}@>t_4>>R^{(S^{**})}@>\varpi>>I_2@>>>0
\end{CD}
\end{equation*}
Now because each polynomial has a unique biggest monomial and each monomial has a unique biggest variable, we can cancel $J'_4$ and $J''_4$ in $J_4$ and replace the central term $R^{(S^{**})}$ by $R^{(I_*)}$. We then obtain the exact sequence
\begin{equation*}
\begin{CD}
 R^{(K\amalg A)}@>t_4>>R^{(I_*)}@>\varpi>>I_2@>>>0
\end{CD}
\end{equation*}
and the proposition is proved.
\end{proof}

\section*{Perspectives}

Beyond this paper, we will use quadratic algebra to show that any quadratic map between modules can be identified with a morphism in a certain monoidal, complete and cocomplete homological category ${\bf M}_R$, whose objects are  of explicit algebraic nature and whose morphisms are families of $R$-linear maps. This allows to carry out constructions with quadratic maps which do not make sense in classical algebra: in particular, they admit kernels,  cokernels, tensor powers etc. Moreover, quadratic algebraic K-theory $K_0^{quad}(R)$ of $R$ can be defined from ${\bf M}_R$. All of this is work in progress and will be presented elsewhere.


\begin{thebibliography}{12}

\bibitem{B} H.J.\ Baues, Quadratic homology, \textit{Trans. Amer. Math. Soc.} {\bf 351} (1999), 429-457.

\bibitem{BHP} H.J.\ Baues, M.\ Hartl and T.\ Pirashvili, Quadratic categories and square rings,  \textit{J.  Pure Appl.\
Algebra} {\bf 122} (1997), 1-40.

\bibitem{BJP1} H.J.\ Baues, M.\ Jibladze and T.\ Pirashvili, Quadratic algebra of square groups,  \textit{Adv. Math.}
 {\bf 217} (3) (2008), 1236-1300.

\bibitem{BJP2} H.J.\ Baues, M.\ Jibladze and T.\ Pirashvili, Third McLane cohomology, \textit{Math. Proc. Cambr. Phil. Soc.} {\bf 144} (2008), 337-367.

\bibitem{BM} H.J.\ Baues, F. Muro, The algebra of secondary homotopy operations in ring spectra, arXiv:0610523v3 (2007).

\bibitem{FS} E.M.\ Friedlander, A.\ Suslin, Cohomology of finite group shemes over a field, \textit{Inv. Math.} {\bf 127} (2) (1997), 209-270.

\bibitem{Diss}
M.\ Hartl,  \newblock {Abelsche Modelle nilpotenter Gruppen, \textit{PhD thesis,} Rheinische  Friedrich-Wilhelms-Universit\"at Bonn},
(1991).

\bibitem{Q3}   M.\ Hartl,   Some successive quotients of group ring filtrations
induced by N-series,     \textit{Comm. in Algebra}  {\bf 23}   (10)  (1995), 3831-3853.

\bibitem{PolProp}    M.\ Hartl,    Polynomiality properties of group extensions with torsion-free abelian kernel,    \textit{J.\ of Algebra}  {\bf 179}  (1996),  380-415.

\bibitem{GoG}     M.\ Hartl,   The nonabelian tensor square and Schur multiplicator of nilpotent groups of class 2,   \textit{J.\ of Algebra}  {\bf 179}  (1996),  416-440. 

\bibitem{QmG}     M.\ Hartl,   Quadratic maps between groups,   arXiv:0707.0371 (2007), to appear in \textit{Georgian Math. J.} 

\bibitem{Rob}   N.\ Roby, Lois polyn\^omes et lois formelles en th\'eorie des modules, \textit{Ann. Scient. \'Ec. Norm. Sup.,} $3^{\rm e}$ s\'erie {\bf 80} (1963), 213-348.

\bibitem{Pa68}  I.\ B.\ S.\ Passi,   Polynomial maps on groups,    \textit{J.
Algebra}   {\bf 9}  (1968),  121-151.

\bibitem{PaQ2} I.\ B.\ S.\ Passi,   Polynomial functors, \textit{Proc. Cambridge Philos. Soc.} {\bf 66} (1969), 505-512.

\bibitem{Wf} R.B. Warfield, Jr., \textit{Nilpotent groups,} LNM Vol.\ 513 (1976), Springer-Verlag, Berlin, Heidelberg, New York.

\end{thebibliography}
\end{document}